\documentclass{article}
\usepackage{graphicx} % Required for inserting images

\usepackage[shortlabels]{enumitem}
\usepackage[hmargin = 1in]{geometry}
\usepackage{color}

\usepackage{amsmath,amssymb,amsthm}
\usepackage{subfig}

\newtheorem{thm}{Theorem}
\newtheorem{conj}[thm]{Conjecture}
\newtheorem{obs}[thm]{Observation}
\newtheorem{cla}[thm]{Claim}
\newtheorem{lem}[thm]{Lemma}
\newtheorem{prob}[thm]{Problem}
\newtheorem{remark}[thm]{Remark}
\newtheorem{oq}[thm]{Open Question}

\usepackage{extarrows}

\usepackage{tikz}
\usetikzlibrary{scopes}

\title{Every connected subcubic graph except the Petersen graph \\ is packing $(1,1,2,2)$-colorable}
\author{Xinmin Hou$^{a,c}$, Xujun Liu$^b$, Xiangyang Wang$^a$\\
	\small$^a$ School of Mathematical Sciences,\\
			\small University of Science and Technology of China, Hefei 230026, Anhui, China\\
    \small$^b$ Department of Applied Mathematics, School of Mathematics and Physics,\\
    \small Xi'an Jiaotong-Liverpool University, Suzhou 215123, Jiangsu, China\\
			\small$^c$ Hefei National Laboratory,\\
\small University of Science and Technology of China, Hefei 230088, Anhui, China\\
\small Email: xmhou@ustc.edu.cn (X. Hou), xujun.liu@xjtlu.edu.cn (X. Liu), wangxiangyang@mail.ustc.edu.cn (X. Wang)
}

\date{}

\begin{document}

\maketitle

\begin{abstract}
For a non-decreasing sequence $S = (s_1, s_2, \ldots, s_k)$ of positive integers, a packing $S$-coloring of a graph $G$ is a partition of $V(G)$ into $V_1, V_2, \ldots, V_k$ such that each $V_i$ has pairwise distance at least $s_i+1$. The packing chromatic number (PCN) of a graph $G$ is the minimum $k$ such that $G$ has a packing $(1,2, \ldots, k)$-coloring. The $1$-subdivision of $G$ is obtained by replacing each edge of $G$ with a path of two edges. In 2016, Gastineau and Togni asked an open question whether the $1$-subdivision of every subcubic graph has PCN at most $5$, and later Bre\v sar, Klav\v zar, Rall, and Wash conjectured it is true. Balogh, Kostochka, and Liu proved the first upper bound of $8$, and it was later improved to $6$ by Liu, Zhang, and Zhang. 

In this paper, we prove that every connected subcubic graph except the Petersen graph is packing $(1,1,2,2)$-colorable. Our result implies a solution to the conjecture of Bre\v sar, Klav\v zar, Rall, and Wash, and answers the question of Gastineau and Togni in the affirmative. Furthermore, our result answers an open question of Kostochka and Liu and solves a conjecture of Liu, Zhang, and Zhang.
\end{abstract}
%The question has been opened for ten years. An $i$-independent set in a graph $G$ is a vertex set whose pairwise distance is at least $i+1$. 
%For a given graph $G$, a packing $S$-coloring for a sequence $(s_1,s_2,\dotsc,s_k)$ of non-decreasing integers is a partition of the vertex set $V(G)$ into subsets $V_1,V_2,\dotsc,V_k$  such that for every pair of distinct vertices $u,v\in V_i$, where $1\le i\le k$, the distance between $u$ and $v$ is at least $s_i + 1$. The packing chromatic number $\chi_p(G)$ of a graph $G$ is defined as the smallest integer $k$ for which $G$ admits a packing $(1,2,\dotsc,k)$-coloring. In this paper, we prove that every subcubic graph $G$ without the Petersen graph has a packing $(1,1,2,2)$-coloring. This also resolves a conjecture proposed by Bre\v{s}ar et al.  That is, the $1$-subdivision of each subcubic graph is packing $(1,2,3,4,5)$-colorable.

\section{Introduction}

An $i$-independent set in a graph $G$ is a vertex set whose pairwise distance is at least $i+1$. For a non-decreasing sequence $S = (s_1, s_2, \ldots, s_k)$ of positive integers, a packing $S$-coloring of a graph $G$ is a partition of $V(G)$ into $V_1, V_2, \ldots, V_k$ such that each $V_i$ is $s_i$-independent. The notion was introduced by Goddard and Xu~\cite{GX1} and has been widely studied (e.g., see~\cite{BKL1, BKL2, BF1, BKRW1, BKR1, FKL1, GT1, GHHHR1, KL1, LW1, LZZ1, MT1, MT2}). Its edge counterpart was introduced by Gastineau and Togni~\cite{GT2} and has received much attention recently (e.g., see~\cite{HLL1,LSS1}). We use exponents to denote repetitions in a sequence. For example, $(1,1,2,2,2)$ is denoted as $(1^2,2^3)$.

%{\color{magenta} We also write $S$ compactly as $(a_1^{k_1}, a_2^{k_2}, \dots,a_m^{k_m})$ where each $a_i$ is a distinct value and $k_i$ its multiplicity.}

The celebrated Four-Color-Theorem states that every planar graph is packing $(1^4)$-colorable. The square of a graph $G$, denote by $G^2$, is obtained from $G$ by adding edges joining vertices of distance exactly two. A {\em square $k$-coloring} of a graph $G$, denoted by $\chi_2(G)$, is the minimum $k$ such that $G^2$ can be properly $k$-colored. Using the packing coloring notation, it is equivalent to a packing $(2^k)$-coloring.  Square coloring of planar graphs was first studied by Wegner~\cite{W1} in 1977, who posed a conjecture that if $G$ is a planar graph of maximum degree at most $\Delta$, then $\chi_2(G) \le 7$ if $\Delta(G) = 3$, $\chi_2(G) \le \Delta(G) + 5$ if $4 \le \Delta(G) \le 7$, and $\chi_2(G) \le \lfloor \frac{\Delta}{2} \rfloor + 1$ if $\Delta \ge 8$. Thomassen~\cite{T1}, and independently Hartke, Jahanbekam, and Thomas~\cite{HJT1} confirmed Wegner's conjecture for $\Delta = 3$. Bousquet, Deschamps, de Meyer, and Pierron~\cite{BDMP1} proved the current best upper bound of $12$ when $\Delta = 4$. For non-planar graphs, Cranston and Kim~\cite{CK1} showed that every connected subcubic graph except the Petersen graph is square $8$-colorable, and this bound is tight. 
%Research of graph coloring theory can be traced back to the proof of celebrated Four-Color-Problem~cite{AH,AHK,}.

We notice that a packing $(1^j,2^k)$-coloring can be viewed as an intermediate coloring between proper coloring and square coloring. Gastineau and Togni~\cite{GT1} proved that every subcubic graph has a packing $(1,1,2,2,2)$-coloring and a packing $(1,2,2,2,2,2,2)$-coloring. They also showed that the Petersen graph has no packing $(1,1,2,2)$-coloring and no packing $(1,2,2,2,2,2)$-coloring. Liu and Wang~\cite{LW1} proved an analogue result of Thomassen~\cite{T1} that every subcubic planar graph has a packing $(1,2,2,2,2,2)$-coloring. Packing colorings of general graphs (non-subcubic) was studied by Choi and Liu~\cite{CL1}, as well as Mortada and Togni~\cite{MT2}. We believe the study of packing $(1^j,2^k)$-coloring can provide insights on how to solve hard problems in square coloring, such as Wegner's conjecture.

%Thomassen~\cite{T1}, and independently Hartke, Jahanbekam, and Thomas~\cite{HJT1} confirmed the above conjecture when $\Delta = 3$. Wegner's conjecture remains open for all $\Delta \ge 4$.
%Recently, Bousquet, Deschamps, de Meyer, and Pierron~\cite{BDMP1} proved the current best upper bound of $12$ when $\Delta = 4$. 
%Amini, Esperet, and van den Heuvel~\cite{AEH1}, and independently Havet, van den Heuvel, McDiarmid, and Reed~\cite{HHMR1,HHMR2} showed that the conjecture holds asymptotically, namely, $\chi_2(G) \le \frac{3}{2} \Delta + o (\Delta)$ as $\Delta \to \infty$. For non-planar graphs, Cranston and Kim~\ref{CK1} showed that every subcubic graph except the Petersen graph is square $8$-colorable. 

The packing chromatic number (PCN) of a graph $G$, denoted by $\chi_p(G)$, is the smallest positive integer such that $G$ has a packing $(1,2, \ldots, k)$-coloring. The notion of PCN was introduced by Goddard, Hedetniemi, Hedetniemi, Harris, and Rall \cite{GHHHR1} under the name broadcast coloring. There are more than 50 papers on the topic (e.g., see~\cite{BKL1, BKL2,BF1,BKR1,BKRW1,FKL1,GT1,KL1,LZZ1,ZM1}). The question whether the PCN of all subcubic graphs was bounded by a constant was first asked by Goddard et al~\cite{GHHHR1} and garnered significant interest in the community (e.g., see~\cite{BKRW1,GT1}). It was finally answered in the negative by Balogh, Kostochka, and Liu~\cite{BKL1}, as well as Bre\v sar and Ferme~\cite{BF1}, using the probabilistic method and explicit construction, respectively. 

The $1$-subdivision of a graph $G$ is obtained by replacing each edge with a path of two edges. Gastineau and Togni~\cite{GT1} asked an open question in 2016 that whether the PCN of the $1$-subdivision of subcubic graphs is bounded by $5$. This is one of the most popular open questions in the topic of packing coloring.

\begin{oq}[Gastineau and Togni~\cite{GT1}]\label{openquestion1}
 Is it true that the $1$-subdivision of subcubic graphs is packing $(1,2,3,4,5)$-colorable? 
\end{oq}

In 2017, Bre\v{s}ar, Klav\v{z}ar, Rall and Wash~\cite{BKRW1} provided a packing $(1,2,3,4,5)$-coloring for the $1$-subdivision of the Petersen graph and conjectured Question~\ref{openquestion1} is true.

\begin{conj}[Bre\v sar, Klav\v zar, Rall, and Wash~\cite{BKRW1}]\label{mainconj}
The $1$-subdivision of every subcubic graph has PCN at most $5$.
\end{conj}

Gastineau and Togni~\cite{GT1} observed that if one can prove every connected subcubic graph except the Petersen graph is packing $(1,1,2,2)$-colorable, then the $1$-subdivision of every subcubic graph has PCN at most $5$. Many progresses has been made toward Question~\ref{openquestion1} and Conjecture~\ref{mainconj}, but it still remains open since 2016. In particular, Bre\v{s}ar, Klav\v{z}ar, Rall and Wash~\cite{BKRW1} proved Conjecture~\ref{mainconj} for all generalized prisms of cycles. Bre\v{s}ar, Kuenzel, Rall~\cite{BKR1} proved the existence of packing $(1,1,2,2)$-coloring for claw-free subcubic graphs and thus confirmed Conjecture~\ref{mainconj} for claw-free cubic graphs. Balogh, Kostochka, and Liu~\cite{BKL1} showed the first upper bound ($8$) toward Conjecture~\ref{mainconj}. Kostochka and Liu~\cite{KL1} proved every subcubic outerplanar graphs has a packing $(1,1,2,4)$-coloring, which confirms Conjecture~\ref{mainconj} for subcubic outerplanar graphs. Furthermore, they posed the following open question.

\begin{oq}[Kostochka and Liu~\cite{KL1}]\label{openquestion2}
Is it true that every subcubic planar graph is packing $(1,1,2,2)$-colorable?    
\end{oq}

Very recently, Liu, Zhang, and Zhang~\cite{LZZ1} proved every subcubic graph is packing $(1,1,2,2,3)$-colorable. Their result implies an upper bound of $6$ for Conjecture~\ref{mainconj}. Furthermore, they posed a stronger conjecture.

\begin{conj}[Liu, Zhang, and Zhang~\cite{LZZ1}]\label{conjecture2}
Every connected subcubic graph except the Petersen graph is packing $(1,1,2,2)$-colorable.    
\end{conj}

Zein and Mortada~\cite{ZM1} further improved the result of Liu, Zhang, and Zhang~\cite{LZZ1} by showing that every subcubic graph has a $(1,1,2,2,k)$ coloring, where $k$ can be any positive integer and is used at most once. They also proved that every $2$-degenerate subcubic graph has a packing $(1,1,2,2)$-coloring. In this paper, we confirm Conjecture~\ref{conjecture2}.

\begin{thm}\label{mainthe}
Every connected subcubic graph $G$ except the Petersen graph is packing $(1,1,2,2)$-colorable. 
\end{thm}

As a corollary of our result, Question~\ref{openquestion1} is answered in the affirmative and Conjecture~\ref{mainconj} is completely solved --- after being opened for a decade.  Question~\ref{openquestion2} is answered in the positive as well. Our result characterizes the Petersen graph as the unique connected subcubic graph whose vertex set cannot be partitioned into two independent sets and two $2$-independent sets (i.e., has no packing $(1,1,2,2)$-coloring).

%In this paper, we mainly prove the following theorem.  This implies that Conjecture \ref{mainconj} is true by Theorem \ref{subthm}. 

\section{Proof of Theorem \ref{mainthe}}
In what follows, we denote by $I_1$ and $I_2$ the two disjoint independent sets, and by $2_a$ and $2_b$ the colors assigned to the two disjoint $2$-independent sets in a possible packing (1,1,2,2)-coloring of $G$.
Our proof follows the framework of Liu, Zhang, and Zhang~\cite{LZZ1}. The main idea is how to choose and further adjust locally two appropriate independent sets so that the remaining vertices can be partitioned into two $2$-independent sets. We assume that $G$ is a connected cubic graph, since every subcubic graph is a subgraph of some larger cubic graph. We first pick two disjoint independent sets, $I_1$ and $I_2$, such that %Furthermore, we assume $G$ is connected.
\renewcommand{\theequation}{\Roman{equation}}
\begin{equation}\label{condition1}
    |I_1| + |I_2|\text{~is the maximum among all choices of $I_1,I_2$}. 
\end{equation}
By Condition~\eqref{condition1}, we observe that any vertex outside $I_1 \cup I_2$ is necessarily adjacent to at least one vertex in $I_1$ and one vertex in $I_2$. 
 Thus, each connected component in $G' = G[V(G)-I_1-I_2]$ is either a single vertex $P_1$ or a path $P_2$ on two vertices as $G$ is cubic.
Among the sets $I_1, I_2$ satisfying Condition~\eqref{condition1}, we further require that they satisfy

\begin{equation}\label{condition2}
    \text{the number of connected components in $G-I_1-I_2$ is minimum.} 
\end{equation}

The graph $H_{I_1,I_2}$ (abbreviated to $H$ if $I_1,I_2$ is clear from the context) is defined to be the graph $H$ with $V(H) = V(G) - I_1 - I_2$ and $E(H) = \{u_1u_2 \text{ }|\text{ }d(u_1,u_2) \le 2, u_1,u_2 \in V(H)\}$.  To better illustrate how each component of $H$ fits into $G$,   let the corresponding graph $G(H)$ of $H$ in $G$ be defined as the graph with vertex set $V(H) \cup \{u \text{ }|\text{ } u \in I_1\cup I_2, v_1, v_2 \in V(H), \text{ and }uv_1, uv_2 \in E(G)\}$ and the edge set $E(G(H))=\{u_1u_2 \text{ }|\text{ } u_1u_2 \in E(G), u_1, u_2 \in V(H)\} \cup \{uv\text{ }|\text{ }uv \in E(G), u \in V(H), v \in V(G(H)) - V(H) \}$. We refer to the vertices in $I_1 \cup I_2$ as black vertices and to the vertices in $V(G) - I_1 - I_2$ as red vertices. Let $u \in I_1 \cup I_2$, say $u \in I_1$, and $v \in V(G) - I_1 - I_2$. A {\em switch} of $u$ and $v$ is defined as the operation of deleting $u$ from $I_1$ and adding $v$ to $I_1$. The following lemma collects lemmas in~\cite{LZZ1} that will be used in our proof.

\begin{lem}\label{collectlemma}
Subject to Conditions \eqref{condition1} and \eqref{condition2}, we have the following statements.

(1) At most one red $P_2$ can be included in a connected component of $G(H)$. 

(2) If three red vertices $u,v,w$ are adjacent to the same vertex $x \in I_1 \cup I_2$, then each vertex $y \in V(G)-\{u,v,w\}$ with $\min\{d(y,u), d(y,v), d(y,w)\} \le 2$ must be in $I_1 \cup I_2$. Furthermore, $u,v,w$ form a triangle component in $H$.

(3) Each component of $H$ is either a tree, an even cycle, or an odd cycle.
\end{lem}
For a cycle component with order 3 in $H$, we have the following three structures in $G$. In fact, since $G$ is cubic, every cycle component of $H$ corresponds to either a $T_1$ or a cycle in $G$ by Lemma \ref{collectlemma} (2).

\begin{figure}[htbp]
\vspace{-5mm}
    \centering
    \subfloat[$T_1.$]{
    \begin{tikzpicture}
    \path (-1.2,0){}
    (1.2,0) {};
        \scoped[every node/.style = {circle, draw,fill, inner sep = 1.6pt}] 
        \draw (0,0) node (x) {}
        +(90:1) node[red] (u) {}
        +(210:1) node[red] (v) {}
        +(-30:1) node[red] (w) {}
        ;
        \draw (x) -- (u)
        (x) -- (v)
        (x) -- (w);
    \end{tikzpicture}
    }\qquad\qquad\qquad
\subfloat[$T_2$.]{
\begin{tikzpicture}
    \scoped[every node/.style = {circle, draw,fill, inner sep = 1.6pt}] 
        \draw (0,0) node[red] (u1) {}
        ++(0:1) node[red] (u2) {}
        ++(72:1) node (u3) {}
        ++(144:1) node[red] (u4) {}
        ++(216:1) node (u5) {}
        ;
        \draw (u2) -- (u3) -- (u4) -- (u5) -- (u1) ;
        \draw[red,thick] (u1) -- (u2);
\end{tikzpicture}
}\qquad\qquad\qquad
\subfloat[$T_3$.]{
\begin{tikzpicture}
    \scoped[every node/.style = {circle, draw,fill, inner sep = 1.6pt}] 
        \draw (0,0) node (u1) {}
        ++(60:1) node[red] (u2) {}
        ++(120:1) node (u3) {}
        ++(180:1) node[red] (u4) {}
        ++(240:1) node (u5) {}
        ++(300:1) node[red] (u6) {}
        ;
        \draw (u1) -- (u2) -- (u3) -- (u4) -- (u5) -- (u6) -- (u1);
\end{tikzpicture}
}
  % \subfloat[$F_1.$]{\includegraphics[width = 3.5cm]{F1.pdf}}
  % \subfloat[$F_2.$]{\includegraphics[height = 2.3cm]{F2.pdf}} 
\caption{The structures in $G$ corresponding to the cycle components of order 3 in $H$. }
\vspace{-3mm}
\end{figure}
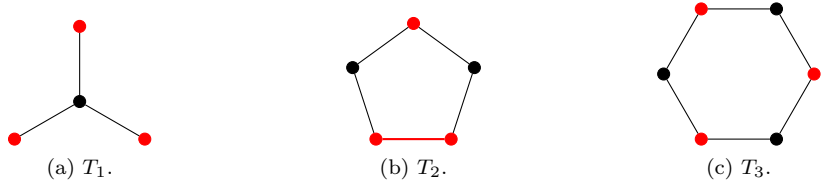

The basic forbidden configurations $F_1$ and $F_2$ shown in Figure~\ref{configurations} play a key role in our proof. By performing some switch operations, we can readily establish the following observation. 

\begin{figure}[ht]
\begin{center}
  \subfloat[$F_1.$]{\includegraphics[width = 3.5cm]{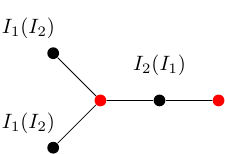}} \hspace{1.in}
  \subfloat[$F_2.$]{\includegraphics[height = 2.3cm]{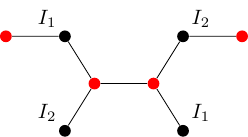}} 
\caption{Configurations $F_1$ and $F_2$.}\label{configurations}
\end{center}
\vspace{-8mm}
\end{figure}

\begin{obs}\label{inadm}
Subject to Conditions~\eqref{condition1} and~\eqref{condition2}, the configurations $F_1$ and $F_2$ cannot occur in $G$.  
\end{obs}

For cycles $C_1, C_2$ in $G$ and a vertex $u \in V(G)$,  define $d(u,C_1) = \min_{v \in V(C_1)} d(u,v)$ as the distance from $u$ to $C_1$, and $d(C_1, C_2) = \min_{u \in V(C_1), v\in V(C_2)} d(u,v)$ as the distance from $C_1$ to $C_2$.

\begin{lem}\label{mainlemma}
Let $C = u_1u_2 \ldots u_k u_1$ be a cycle component in $H$ and its corresponding graph in $G$ is not a $T_1$. Suppose its corresponding cycle $C'$ in $G$ is $u_1v_1u_2v_2 \ldots u_kv_ku_1$ (when there is no red $P_2$) or $u_1v_1u_2v_2 \ldots u_ku_1$ (when $u_1u_k$ is the unique red $P_2$). The following statements hold.
%$G(C)$ be its corresponding graph in $G(H)$.

(1) The vertices $\{v_1, \ldots, v_k\}$ must be entirely contained in either $I_1$ or $I_2$.
    
(2) $C'$ has no chord in $G$.

(3) Let $x \notin V(C')$ be a vertex in $G$. If $d(x,C') \le 2$, then $x \in I_1 \cup I_2$. Suppose $v_1, \ldots, v_k \in I_1$. If $d(x,C') = 1$, then $x \in I_2$; if $d(x,C') = 2$, then $x \in I_1$. 
%if $d(x,C')=d(x,y)=2$ for some $y \in \{v_1, \ldots, v_k\}$, then $x \in I_1$. % except $x \in \{u_1, \ldots, u_k\}$.

%   (1)If $C_0^*$ is a cycle, then the outer-neighbour of a red isolated and other red vertices can not be the same. 
   
  % (2)The consecutive vertices can not have the same outer-neighbour. 
\end{lem}
\begin{proof}
(1) Suppose $v_1 \in I_1$ and $v_2 \in I_2$. Let $N(u_2) = \{v_1, v_2, w_2\}$. Without loss of generality, we may assume $w_2 \in I_1$. This creates a configuration $F_1$ in $G$, yielding a contradiction. Hence, $v_1 \in I_1$ implies $v_2 \in I_1$. By a similar argument, we can conclude inductively that $v_1, \ldots, v_k \in I_1$.

(2) By (1), we may assume $\{v_1, \ldots, v_k\} \subseteq I_1$. This implies  $v_i v_j\notin E(G)$ for $i \neq j$, and also forbids any chord $u_i v_j$ otherwise adding $u_i$ to $I_2$ would violate Condition~\eqref{condition1}. Likewise, an edge $u_i u_j$ (excluding $u_1 u_k$ when $C' = u_1 v_1 u_2 v_2 \ldots u_k u_1$) is prohibited, as it would similarly allow the addition of $u_i$ to $I_2$, contradicting Condition~\eqref{condition1}.

(3) Let $x \in V(G)$ and $y \in V(C')$. By Lemma~\ref{collectlemma}, if $d(x, C') = d(x,y) = 1$, then $x \in I_1 \cup I_2$. By (1), if $y \in \{v_1, \ldots, v_k\}$, then $x \in I_2$; otherwise, $y \in \{u_1, \ldots, u_k\}$, then $x \in I_2$ since otherwise we can add $x$ to $I_2$ and it contradicts Condition \eqref{condition1}. Moreover, if $ d(x, C') = d(x,y) = 2$ and $y \in \{u_1,\dotsc,u_k\}$, then $x \in I_1$  otherwise we have an forbidden configuration $F_1$ or $F_2$. For   $y \in \{v_1,\dotsc,v_k\}$, if there is no red $P_2$, we construct a new independent set $I_1$ by rearranging the vertices as $I_1\setminus\{v_1,\dotsc,v_k\}\cup \{u_1,\dotsc,u_k\}$. Otherwise, if a red $P_2$ exists, we instead form $I_1$ by taking $I_1\setminus\{v_1,\dotsc,v_{k - 1}\}\cup \{u_2,\dotsc,u_{k}\}$. Using a similar argument, we  then conclude that $x\in I_1$.
\end{proof}

%\noindent\textbf{Remark}: 

\begin{remark}\label{remark}
In the proof of Lemma~\ref{mainlemma} (3), we used a series of switch operations in cycle $C'$. We conclude it here. Suppose the setup of cycle $C$ and $C'$ is the same as Lemma~\ref{mainlemma}. In the case there is no red $P_2$, we can delete $v_1, \ldots, v_k$ from $I_1$ and add $u_1, \ldots, u_k$ to $I_1$. In this way, we formed a new $I_1$ satisfying Conditions~\eqref{condition1} and~\eqref{condition2}. In the case there is a red $P_2$, for each edge $e = xy \in E(C')$, we can perform a series of switch operations so that $e$ is the red $P_2$.
%It should be noted that we can always rearrange the vertices of $I_1$
%in the cycle $C'$ (maximum independent set) so that conditions \eqref{condition1} and \eqref{condition2} continue to hold, with all vertices at distance 1 belonging to $I_2$
%and all black vertices in this cycle belonging to $I_1$. And this preserves the cycle component of $H$.    
\end{remark}

Let $C$ be a cycle component in $H$. We have the following properties about switches.

    (1) If $C$ corresponds to a $T_1$ in $G$, let $V(C) = \{u,v,w\}$ and let $x$ be the vertex in $I_1 \cup I_2$ adjacent to each of $u, v$, and $w$. Now suppose $u'$ is a neighbour of $u$ other than $x$ with $u' \in I_i$ for some $i \in {1,2}$.  Then, by placing $x$ in $I_{3-i}$ and switching $u$ and $u'$, Conditions \eqref{condition1} and \eqref{condition2} are preserved; for otherwise, the original assumption would be violated.

    (2) If $C$ corresponds to a cycle $C'$ in $G$, then for every red $P_1 = u$ in $C'$ with a neighbour $u'$ outside $C'$, switching $u$ and $u'$ preserves Conditions \eqref{condition1} and \eqref{condition2}.

\vspace{1em}

By Lemma~\ref{collectlemma}, every component of $H$ is either a path or a cycle. Paths and even cycles are bipartite and therefore can be colored with $2_a$ and $2_b$. Therefore, it remains to consider the odd cycle components of $H$. Let $N(H)$ be the number of cycle components in $H$. Subject to Conditions \eqref{condition1} and \eqref{condition2}, we further require
\begin{equation}\label{condition3}
\text{$N(H)$ is minimized among all choices of $I_1, I_2$.} 
\end{equation}

We will show that there is at most one cycle component in $H$. 

\begin{lem}
    $N(H)\le 1$. 
\end{lem}

\begin{proof}
 Suppose $N(H)\ge 2$. Let $C_1, C_2$ be two cycle components in $H$. Let \[f(C_1,C_2) := \min_{u\in V(C_1), v\in V(C_2)}d_G(u,v).\]
By the definition of $H$, $f(C_1,C_2)\ge 3$. We further define $f(H)$ to be the minimum of $f(C_1, C_2)$, where $C_1, C_2$ are two cycle components in $H$.  Subject to Conditions~\eqref{condition1},~\eqref{condition2}, and~\eqref{condition3}, we require $f(H)$ is minimized among all choices of $I_1, I_2$. Let $C_1, C_2$ be two cycle components in $H$ such that $f(C_1, C_2) = f(H)$, $C_i'$ be the corresponding structure in $G$ ($T_1$ or a cycle), $i = 1,2$. Let $u_i\in V(C_i'), i = 1,2$, such that $d(C_1',C_2') = d(u_1,u_2)$. By Remark~\ref{remark}, we can perform a series of switch operations so that $u_1, u_2 \in G-I_1-I_2$ if they are not already. Therefore, we may assume $f(C_1,C_2) = d(C_1',C_2') = d(u_1,u_2)$. Let $u_1v_1\dotsm v_ku_2(k\ge 2)$ be the shortest path between $u_1$ and $u_2$. 
By Remark \ref{remark}, we may assume that $u_1$ is a  red $P_1$. Moreover, we assume that $v_1$ belongs to $I_2$. In the case where $C_1$ is a $T_1$, we can place the black vertex in $T_1$  in $I_1$.  We then switch  $u_1$ and $v_1$, and $V(C_1) \setminus \{u_1\}$ cannot form a cycle component of $H$. Under Condition \eqref{condition3},  in order to preserve the minimality of the number of cycle components in $H$, $v_1$ must participate in forming a new cycle component in $H$.  However, this would  reduce $f(H)$ by at least one,  contradicting the minimality of $f(H)$ with Conditions \eqref{condition1},\eqref{condition2}, and \eqref{condition3}. Therefore, we conclude that $N(H)\le 1$. 
%We can further assume $u_i\in V(C_i),i = 1,2$. Suppose not, say $u_1 \notin V(C_1)$ and the the distance $d(C_1', C_2')$ is realized by the path $u_1x_1 \ldots x_ku_2$. 
% Otherwise,  $C_i'$ is a  cycle in $G$ for some $i$. By applying the remark following Lemma~\ref{mainlemma}, we can rearrange the independent set within the cycle so that $u_i\in V(C_i),i = 1,2$.  This results in a strict decrease of $f(H)$ by at least 1, contradicting the minimality of $f(H)$. 
% Let $C_1$ and $C_2$ be two cycle components  of $H$, $C_1^*$ and $C_2^*$ be the corresponding subgraph of $G$, which is either a $T_1$ or a cycle. Let $d$ denote the minimum possible distance between $C_1^*$ and $C_2^*$, and let $u \in C_1$ and $v \in C_2$ be the end vertices of a shortest path between $C_1$ and $C_2$. Without loss of generality, assume that $u$ and $v$ are isolated red vertices with $\gamma(u) = r$. It is evident that $d \geq 3$.
% \textbf{Case 1:} $d_G(H) \ge 4$. We delete $w_1$ from $I_2$ and add $u_1$ to $I_2$. Let the new $I_2$ be called $I_2'$. We claim $w_1$ must be contained in some cycle of $H_{I_1, I_2'}$. Otherwise, since the remaining red vertices in $V(C_1)-u_1$ are also cycle-free, this implies $N(H(I_1, I_2')) < N(H(I_1, I_2))$ and therefore contradict Condition~\ref{condition3}. Let $w_1$ be contained in a new cycle, say $C_3$, in $H_{I_1, I_2'}$. This implies $d_G(H(I_1, I_2')) \le d_G(w_1, u_2) < d_G(H(I_1, I_2))$, which contradicts Condition~\ref{condition4}.  
% \textbf{Case 2:} $d_G(H) = 3$.
\end{proof}

Suppose $N(H) = 1$, let $C_0$ be the unique cycle component of $H$, and let $C'$ be the corresponding structure in $G$. We assume that $G$ is not the Petersen graph. A red $P_1 = u$ in $C'$ is said to have the {\em good property} if, by applying the switch operation with a neighbour $u'$ of $u$ outside $C'$, the vertex $u'$ must lie in a new cycle component of $H$ that does not involve $C_0 - {u}$. If $u$ does not have the good property, then by Condition \eqref{condition3} and the fact that $N(H) = 1$,  $u'$ and $C_0 - u$ must form  a new cycle component in $H$.

\begin{figure}[ht]
\vspace{-10mm}
\begin{center}
  \includegraphics[scale=0.45]{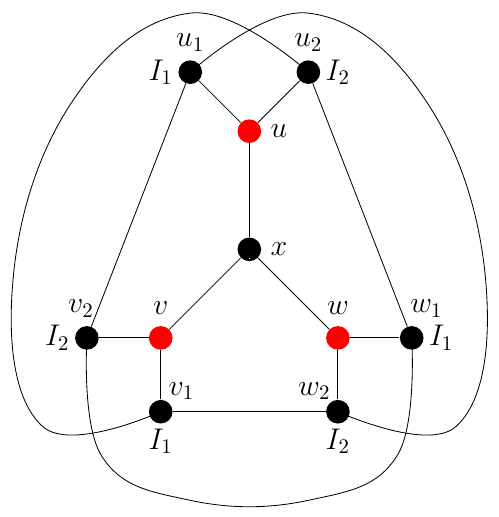} \hspace{2.5cm}
  \includegraphics[scale=0.65]{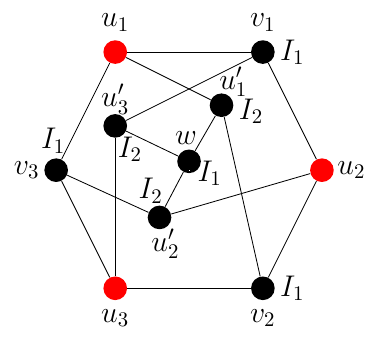} 
\caption{Form the Petersen graph.}
\end{center}
\vspace{-8mm}
\end{figure}

\begin{lem}\label{anat1}
If $C_0$  corresponds to a $T_1$ in $G$, let $V(C_0) = \{u,v,w\}$ and let $x$ be the vertex in $I_1 \cup I_2$ that is adjacent to each of $u, v$ and $w$. Then at least two of $u,v,w$ have the good property.
\end{lem}

\begin{proof}
Let $N(u) = \{u_1, u_2, x\}, N(v) = \{v_1,v_2,x\}, N(w) = \{w_1,w_2,x\}$. Since $x$ can be assigned in either $I_1$ or $I_2$, we can assume $u_1, v_1, w_1 \in I_1$ and $u_2, v_2, w_2 \in I_2$ since each of $u,v$ and $w$ must be adjacent to one vertex in $I_1$ and one vertex in  $I_2$. Moreover, $N(u_1)-u$ and $N(u_2)-u$ must be black, since otherwise it creates a configuration $F_1$, contradicting Observation~\ref{inadm}. The same argument also applies to $v_1, v_2, w_1, w_2$. Therefore,  $u_1$, $v_1$, $w_1$, $u_2$, $v_2$, and $w_2$ are all distinct vertices. Suppose at least two of $u,v,w$, say $u$ and $v$ do not have the good property. That is,  each of them will form a cycle with the remaining two of $u,v,w$ after the switch operation with their neighbours (not $x$). In other words, after the switch of $u$  and $u_1$ , $u_1$  will form a cycle with $v, w$  in $H$. Therefore, $d(u_1,v) = d(u_1,w) = 2$. Thus, $u_1v_2, u_1w_2 \in E(G)$ since $u_1,v_1,w_1\in I_1$. Similarly, consider the switches of $u$ and $v_2$, $v$ and $v_1,v_2$, we have $u_2v_1,u_2w_1,  v_1w_2, v_2w_1 \in E(G)$. However, this forces $G$ to be the Petersen graph, a contradiction. Thus, at least two of $u,v,w$ have the good property.
\end{proof}

\begin{remark}\label{t1property}
 From the proof we see that the neighbours of three red vertices (except $x$) of a $T_1$ are all distinct.   
\end{remark}

Now, suppose  $C'$ is $u_1v_1u_2v_2 \ldots u_kv_ku_1$ (when there is no red $P_2$) or $u_1v_1u_2v_2 \ldots u_ku_1$ (when $u_1u_k$ is the unique red $P_2$) where $u_i(1\le i\le k)$ are red vertices in $G$ and $k = |C_0|$. 

\begin{lem}\label{nosameneighbour}
 The following statements hold.
%If $C_0$  corresponds to a cycle $C'$ in $G$, then

(1) $u_i$ and $u_j$ do not have a common neighbour outside $C'$, where $1 \le i \neq j \le k$, except the case when $i = 1, j = k$, and $u_1u_k$ is the red $P_2$ in $C$;

%the neighbour not on $C'$ of a red $P_1$  and that of another red vertex cannot be the same;
   
(2) Two vertices in $C'$ that are at distance 2 cannot have the same neighbour outside $C'$;
   
(3) The consecutive vertices of $C'$ cannot have the same neighbour;  

(4) When a red $P_1$ and a black vertex in $C'$ share a common neighbour outside $C'$, then $|C_0| = 3$.   
\end{lem}
\begin{proof}
(1) If $i\ne 1$($j\ne k$), this would create a configuration $F_1$ formed by vertices $v_{i - 1},u_i,v_i,u_j$($v_{j - 1},u_j,v_j,u_i$) with the common neighbour by Lemma~\ref{mainlemma}, contradicting Observation~\ref{inadm}.
    
(2) Suppose $u$ and $u'$ are two vertices in $C'$ with $d_{C'}(u,u') = 2$ that share a common neighbour outside $C'$. Since there is a red $P_1$ in $C'$, by  Remark~\ref{remark}, we may assume that $u$ is a red $P_1$.  Then $u'$ must also be a red vertex, which contradicts (1).

% we assume $C' = u_1v_1u_2v_2\dotsm u_ku_1$ otherwise $C' = u_1v_1u_2v_2\dotsm u_kv_ku_1$, where each $u_i(1\le i\le k)$ is a red vertex.
(3) If $C'$ has no red $P_2$ (contains a red $P_2$), then by Remark~\ref{remark}, we only need to consider the case where $u_1$ and $v_k$ ($u_1$ and $u_k$) share a common neighbour $w$. We switch $u_1$ and $w$, and $u_2 \ldots u_k$ form a path in $H$. 
%By Lemma~\ref{mainlemma}, this switching operation preserves Conditions \eqref{condition1} and  $\eqref{condition2}$. 
By Condition~\eqref{condition3} and $N(H) = 1 $, $w$ must belong to a new cycle component in $H$.  Since $d(w,u_k) = 2$ ($d(w,u_k) = 1$), $u_2\dotsm u_kw$ is a path in $H$. Since every vertex in a cycle component has degree $2$, the other edge incident to $w$ in this component must be $wu_2\in E(H)$. Thus, $d(w,u_2)\le 2$. By (1), (2), and the property that all vertices at distance $1$ from $C'$ lie in the same independent set, we know $wv_1 \notin E(G)$ and must have $wv_2\in E(G)$. Consequently, $w$, $u_2$, and $u_3$ are all adjacent to the same vertex $v_2$ (forms a $T_1$), forming a cycle component $u_2u_3wu_2$ of $H$ of order $3$. Hence we conclude that $k=3$. However, $v_2$ and $v_3$ ($u_1$) in $C'$ are at distance $2$ and share a neighbour $w$ outside $C'$, which contradicts (2).

%If $C'$ contains a red $P_2$, then by Remark~\ref{remark}, we may assume that $u_1$ and $u_k$ share a common neighbour $w$.  We now switch  $u_1$ and $w$, then, by Lemma~\ref{mainlemma}, $I_1, I_2$ also satisfy Conditions \eqref{condition1} and \eqref{condition2}. By Condition~\eqref{condition3} and $N(H) = 1 $, $w$ must belong to a new cycle component in $H$.  Since $d(w,u_k) =  1$, $u_2\dotsm u_kw$ is a path in $H$. Because every vertex in a cycle component has degree $2$, the other edge incident to $w$ in this component must be $wu_2\in E(H)$. Thus, $d(w,u_2)\le 2$.  By (1), (2), and the property that all vertices at distance $1$ from $C'$ lie in the same independent set, we deduce $wv_2\in E(G)$. Consequently, $w$, $u_2$ and $u_3$ are all adjacent to the same vertex $v_2$ (form a $T_1$), forming a cycle component $u_2u_3wu_2$ of $H$ of order 3. Hence, $k = 3$. However, $v_2$ and $u_1$ on $C'$ are at distance $2$ and share a neighbour $w$ outside $C'$, which contradicts (2).

(4) Let this red $P_1$ be denoted by $u$,  and suppose it shares a common neighbour $w$ outside $C'$ with a black vertex $v$. Then $v$ is adjacent to two red vertices $u_1$ and $u_2$ with $u\notin \{u_1,u_2\}$ by (3). Switching $u$ and $w$ would  then create a $T_1$ in $G$.  By Lemma \ref{collectlemma}, $|C_0| = 3$ .
\end{proof}

\begin{lem}\label{anat3}
If $C_0$  corresponds to a $T_3$ in $G$, then at least two of $\{u_1, u_2, u_3\}$ have the good property.
%Then at least two of $u_1,u_2,u_3$ can form a new cycle in $H$ not involving the other two remaining vertices of $u_1,u_2,u_3$ via the switch operation with 
%When $C_0$ is formed by a $T_3$.   Consider two fixed red vertices $u_1$ and $u_2$ in $C_0^*$ with $\gamma(u_1) = r_1$ and $\gamma(u_2) = r_2$, and let the third red vertex be $u_3$ with $\gamma(u_3) = r$.  If we push out any of  $r_1,r_2$,   the three labels $r_1,r_2$ and $r_3$ again form a cycle of $H$, then $G$ is the Petersen graph. 
\end{lem}

\begin{proof}
By Lemma~\ref{mainlemma}, we can assume $v_1, v_2, v_3 \in I_1$. Let $N(u_1) = \{v_1, v_3, u_1'\}$, $N(u_2) = \{v_1, v_2, u_2'\}$, and $N(u_3) = \{v_2, v_3, u_3'\}$. Suppose $u_1$ and $u_2$ do not have the good property. In other words, $u_1'$ ($u_2'$) forms a cycle component in $H$ with $u_2, u_3$ ($u_1, u_3$) after the switch operation with $u_1$ ($u_2$). By Lemma \ref{mainlemma}, we know that $u_1',u_2'$ and $u_3'$ belong to $I_2$ and, therefore, are pairwise nonadjacent. 

By assumption, we have $d(u_1',u_2),d(u_1',u_3)\le 2$ and $d(u_2',u_1),d(u_2',u_3)\le 2$.  By Lemma \ref{nosameneighbour}, we have $u_1'v_2 \in E(G)$ and $u_2'v_3\in E(G)$. Moreover, when we switch $u_1$ and $u_1'$, vertices $u_1'$, $u_2$ and  $u_3$ form a $T_1$ in $G$. By Lemma \ref{collectlemma}, there are no other red vertices other than $u_2,u_3$ within distance two from $u_1'$.

We claim that $v_1u_3' \in E(G)$. To see this, we reassign the vertices as follows: assign $u_1$ and $v_2$ to $I_2$, assign $u_2$ to $I_1$, and remove $v_1$ from $I_1$ and $u_1'$ from $I_2$, while leaving all other vertices unchanged. Then $I_1$ and $I_2$ remain two disjoint independent sets, and again satisfy Conditions \eqref{condition1} and \eqref{condition2}. Moreover, we have $d(u_3,u_1') = d(v_1,u_1') = 2$. We further note that, by Lemma~\ref{mainlemma}, all vertices outside of $C'$ and within distance two from $v_1, u_3$ are in $I_1 \cup I_2$. To preserve the unique cycle component of $H$, we must have $d(v_1,u_3) = 2$.  

We now perform two switches, i.e., the pairs $(u_1, u_1')$ and $(u_2,u_2')$. Then $I_1$ and $I_2$ remain two disjoint independent sets, and again satisfy Conditions \eqref{condition1} and \eqref{condition2}. Moreover, we have $d(u_3,u_1') = d(u_3,u_2') = 2$. Since there are no other red vertices other than $u_2',u_3$ within distance two from $u_1'$, $u_1'$ can only form a cycle with $u_2', u_3$. To preserve the unique cycle component of $H$, we must have $d(u_1',u_2') = 2$. By symmetry, we can switch $(u_1, u_1')$ and $(u_3,u_3')$, and we obtain $d(u_1',u_3') = 2$. Therefore, there exists another vertex $w$  adjacent to $u_1',u_2'$, and $u_3'$. However, this forces $G$ to be the Petersen graph,  a contradiction. 
%Therefore, at least two of $\{u_1, u_2, u_3\}$ have the good property. reassign the vertices as follows: assign $u_1$ and $u_2$ to $I_2$, and remove $u_1'$ and $u_2'$ from $I_2$.
\end{proof}

\begin{lem}\label{ord4}
If either the order of $C_0$ is at least $4$, or $C_0$ corresponds to a $T_2$ in $G$. Then every red $P_1$ in $C'$ has the good property.  
\end{lem}

\begin{proof}
Suppose $u$ is a red $P_1$, then $C_0 - u$ is a  path $P$. Let $u'$ be the neighbor outside $C'$ of $u$. By Lemma~\ref{nosameneighbour}, the distance between  $u'$ and each of the two red end vertices of $P$ is at least 3. Since $N(H) = 1$, $u'$ must be in a new cycle component in $H$ when we switch $u$ and  $u'$.  Therefore, consider the two possible neighbours in $H$ of $u'$, the new cycle must be formed by $u'$ together with another path component of $H$.
\end{proof}

For a path component $P$ in $H$, let $Y$ be the set of vertices in $G$ whose distance to the red vertices in $P$ is less than 3. We then define $P^* := G[Y]$. A black vertex $u$ in 
$P^*$ is called an {\em interface} of $P$ if 

(1) $u$ lies at a distance exactly 2 from the set of red vertices; 

(2) under some choice of $I_1$ and $I_2$ satisfying Conditions~\eqref{condition1},~\eqref{condition2}, and~\eqref{condition3},  $u$ is red with all other vertices of $P^*$ unchanged;

(3) $u$ and the other  red vertices in $P^*$ form a cycle in $H$.

An {\em inner vertex} of $P^*$ is a black vertex adjacent to two red vertices, one of which is a red $P_1$. 

\begin{lem}\label{samein}
    The inner vertices of $P^*$ must be in the same independent set.
\end{lem}
\begin{proof}
If two inner vertices are adjacent to a common red $P_1$, then they must belong to the same independent set otherwise  an $F_1$ would appear in $G$. If two inner vertices are adjacent to the end vertices of a red $P_2$, then they must belong to the same independent set, since otherwise an $F_2$ would appear in $G$. Furthermore, since $P$ forms a path component of $H$, this forces all inner vertices into the same independent set. 
\end{proof}

We analyze properties of interfaces under the following three cases.

\begin{itemize}
\item When $|P|\ge 3$, $P^*$ must have inner vertices. By Lemma \ref{samein}, we assume the inner vertices in $P^*$ are all in $I_1$. If a black vertex $x$ is an interface of $P$. By the definition of interface, $x$ can be turned into red with all other vertices in $P^*$ unchanged. We turn $x$ into red to form a cycle $C$ in $H$. By Lemma \ref{mainlemma}, and since every red vertex must be adjacent to a vertex in $I_2$, it follows that $x$ is adjacent to the end vertices of $P$ via two non‑inner black vertices $z_1, z_2 \in I_1$. Furthermore, $C$ is the unique cycle component in $H$. Consequently, by Lemma \ref{nosameneighbour}, $P$ has only one interface (see the first graph in Figure~4). 
    
\item Consider the case $|P| = 2$, where $P^*$ contains a red $P_2 = uv$. If $u$ and $v$ share a common neighbour $w$, then we argue that the other neighbour $w'$ of $w$ is not an interface. Suppose not, $w'$ is an interface. By the definition of interface, $w'$ can be turned into red with all other vertices in $P^*$ unchanged. In this case, $w$ could be placed in either $I_1$ or $I_2$, which is a contradiction with Condition~\eqref{condition1} since we can add $u$ to $I_1$ or $I_2$. Consequently, if $P^*$ contains two interfaces, Lemma \ref{mainlemma} implies they must be the blue vertices depicted in the second graph in Figure \ref{interface}.

 %To see this, assume for contradiction that there is a valid choice of $I_1$ and $I_2$ satisfying Conditions \eqref{condition1}-\eqref{condition3} that turns $w$ red and leaves the rest of $P^*$ unchanged.

\item In the case $|P| = 2$, where $P^*$ contains  two red $P_1$s, namely $u$ and $v$, assume that the inner vertex $x$ of $P^*$ are in $I_1$. Suppose there exists an interface $w$ of $P$ adjacent to  $x$. By the definition of interface, $w$ can be turned into red with all other vertices in $P^*$ unchanged. Then $u$, $v$, and $w$ induce a $T_1$ in $G$. By Remark~\ref{t1property}, the neighbours of $u$, $v$, and $w$ (other than $x$) are all distinct. In the presence of a second interface $w'$, this $w'$ must be connected to $u$ and $v$ through two  non-inner black vertices in $I_1$ by Lemma \ref{mainlemma}. It is illustrated by the blue vertices in the third graph in Figure 4. Otherwise, if no such an interface  $w$ exists, Lemma \ref{mainlemma} implies there is at most one interface, which connects to $u$ and $v$ via two non-inner vertices belonging to $I_1$. 
\end{itemize}

\begin{figure}[htbp]
\vspace{-6mm}
    \centering

    \hspace{-5mm}
    \subfloat{
\begin{tikzpicture}[scale = 0.8]
    \scoped[every node/.style = {circle, draw,fill, inner sep = 1.6pt}] 
    \draw (0,0) node[red] (r1) {}
    (0,2) node[red] (r2) {}
    (0,-0.8)
    (2,1) node[blue] (b) {}
    (1,2.4) node[label = right:$I_1$] (x) {}
    (1,-.4) node[label = right:$I_1$] (y) {}
    ;
    \draw[dash dot] (r1) -- (r2);
    \draw (r1) -- (y) -- (b) -- (x) -- (r2)
    ;
\end{tikzpicture}
}
\hspace{2mm}
%\qquad
\subfloat{
\begin{tikzpicture}[scale = .8]
    \scoped[every node/.style = {circle, draw,fill, inner sep = 1.6pt}] 
        \draw (0,-1.3) 
        (0,0) node[red] (r1) {}
        +(-18:1) node[label = below:$I_1$] (x2) {} +(-162:1) node[label = below:$I_2$] (y2) {}
        (0,1) node[red] (r2){}
        +(18:1) node[label = above:$I_1$] (x1) {} +(162:1) node[label = above:$I_2$] (y1) {}
        (2,0.5) node[blue] (b1) {}
        (-2,0.5) node[blue] (b2) {};
        \draw (r1) -- (r2) -- (x1) -- (b1) -- (x2) -- (r1)
        (r1) -- (y2) -- (b2) -- (y1) -- (r2)
        ;
\end{tikzpicture}
}
%\hspace{1mm}
%\qquad
    \subfloat{
    \begin{tikzpicture}[scale = .8]
    \path (-1.2,0){}
    (1.2,0) {};
        \scoped[every node/.style = {circle, draw,fill, inner sep = 1.6pt}] 
        \draw (0,0) node[label = right:$I_1$] (x) {} 
         (-1,0) node[blue] (b1) {}
         (0,1) node[red] (r1) {} +(150:1) node[label =   left:$I_2$] (rr1){}
         (0,-1) node[red] (r2) {} +(-150:1) node[label = left:$I_2$] (rr2){}
        (1,1) node[label = above:$I_1$] (y) {}
        (1,-1) node[label = below:$I_1$] (z) {}
        (2,0) node[blue] (b2) {}
        ;
        \draw (x) -- (b1)
        (x) -- (r1) -- (rr1)
        (x) -- (r2) -- (rr2)
        (r1) -- (y) -- (b2) -- (z) -- (r2)
        ;
    \end{tikzpicture}
    }
\hspace{2mm}
\quad
    \subfloat{
    \begin{tikzpicture}[scale = .8]
        \draw[red] (0,0) circle [radius = 1cm];
         \scoped[every node/.style = {circle, draw,fill, inner sep = 1.6pt}] 
         \draw  (2cm,1cm) node[label = below:$u$] (u1) {}
         (2cm,-1cm) node[label = below:$v$] (u2) {}
         (3cm,0cm) node[label = below right:$w$] (v) {}
         (4cm,0cm) node[red,label = below:$z$] (w) {}
         (30:1cm) node[red,label = left:$u_1$] (uu1) {}
         (-30:1cm) node[red,label = left:$u_2$] (uu2) {}
         ;
         \draw (uu1) -- (u1) -- (v) -- (w);
         \draw (uu2) -- (u2) -- (v);
    \end{tikzpicture}
    }

\caption{The interfaces of path component of $H$ and the structure used in Lemma \ref{admissthird}.}\label{interface}
\vspace{-3mm}
\end{figure}
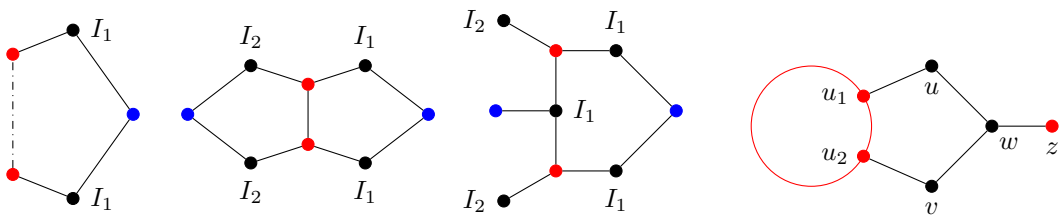

\begin{lem}\label{admissthird}
Suppose $C$ is the unique cycle component of $H$. Let $C'$ be the corresponding structure in $G$ (a $T_1$ or a cycle). If $C'$ is a $T_1$, let $U$ be its three red vertices; otherwise, let $U = V(C')$. For a red vertex $z \notin V(C)$, each of its neighbours $w$ in $G$ cannot form a path $u_1uwvu_2$ of length $4$, where $u_1, u_2 \in U$ and $u,v \notin U$ (see Figure \ref{interface}).
%there is no vertex $w \in N(z)$ such that there exist two vertices $u_1,u_2$ in $C'$ and two distinct vertices $u,v$ for which $u_1uw$ and $u_2vw$ are both paths of length 2 in $G$(see Figure \ref{thirdred}). 
\end{lem}

\iffalse

\begin{figure}[htbp]
    \centering
    \subfloat{
    \begin{tikzpicture}
        \draw[red] (0,0) circle [radius = 1cm];
         \scoped[every node/.style = {circle, draw,fill, inner sep = 1.6pt}] 
         \draw  (2cm,1cm) node[label = below:$u$] (u1) {}
         (2cm,-1cm) node[label = below:$v$] (u2) {}
         (3cm,0cm) node[label = below right:$w$] (v) {}
         (4cm,0cm) node[red,label = below:$z$] (w) {}
         (30:1cm) node[red,label = left:$u_1$] (uu1) {}
         (-30:1cm) node[red,label = left:$u_2$] (uu2) {}
         ;
         \draw (uu1) -- (u1) -- (v) -- (w);
         \draw (uu2) -- (u2) -- (v);
    \end{tikzpicture}
    }
    \caption{The structure used in Lemma \ref{admissthird}.}\label{thirdred}
\end{figure}

\fi

\begin{proof}
First of all, we know that $d(z,C')\ge 3$ by Lemma \ref{collectlemma} and \ref{mainlemma}. We can assume $u_1$ is a red $P_1$ by Remark~\ref{remark}. We switch $u_1$ and $u$. By Condition~\eqref{condition3}, $N(H) = 1$, and $d(u,z) = 2$, vertex $u$ together with the path component $P$ which contains $z$ forms a new cycle component of $H$. By a similar argument for $u_2$, we conclude that $u$ and $v$ are both interfaces of $P$ (it can only be the two cases shown in the middle two graphs in Figure~\ref{interface}),  and $z$ is a red end vertex of $P$. However, considering the three neighbours of each red end vertex in the middle two graphs in Figure~\ref{interface}, no suitable vertex can serve as $w$ by Lemma \ref{mainlemma} and \ref{nosameneighbour}. 
\end{proof}
%, as $d(u,z) = 2$ and the unique cycle component of $H$ must be preserved.

\noindent \textbf{Proof of Theorem \ref{mainthe}:}
Recall $C_0$ is the unique cycle component of $H$. Let $C_0'$ be the corresponding structure in $G$ (a $T_1$ or a cycle). By Lemma \ref{anat1}, \ref{anat3} and \ref{ord4}, there is a red $P_1 = x_0$ in $C_0$ which has the good property. Let $y_0$ be the neighbour of $x_0$ outside $C_0'$ such that it forms a new cycle $C_1$ with another path component $L_1$ in $H$ when we switch $x_0$ and $y_0$. We switch $x_0$ and $y_0$ to obtain $C_1$. We carry out the following procedure. At step $i\ge 1$, let $C_i'$ be the corresponding structure of $C_i$. %Note that 

\begin{enumerate}[(1)]
\item  If $C_i'$ is not a $T_2$, then by Lemma~\ref{anat1},~\ref{anat3} and~\ref{ord4}, there are at least two $P_1$s with good property. We switch one such red $P_1$, say $x_i$ (distinct from $y_{i-1}$), with its neighbour $y_i$ outside $C_i'$. We obtain a new cycle component $C_{i+1}$ together with a path component $L_{i+1}$ in $H$, leaving $C_i - x_i$ as a new path component $L_i'$ in $H$.

\item If $C_i'$ is a $T_2$, by Remark~\ref{remark}, we can rearrange the independent sets in $C_i'$ such that $y_{i-1}$ is in a red $P_2$, leaving another vertex $x_i$ as a red $P_1$. We switch $x_i$ with its  neighbour $y_i$ outside $C_i'$.  This would enable $y_i$ to form, by Lemma \ref{ord4}, a new cycle component $C_{i+1}$ with a path component $L_{i+1}$ in $H$, leaving $C_i - x_i$ as a new path component $L_i'$ in $H$. Note that before we  do the rearrangement, $x_i$ was part of a red $P_2$.
\end{enumerate}

%such that after switching $x_0$  with its neighbour  $y_0$ outside $C_0'$,  $y_0$ together with another path component $P_1$ of $H$ forms a new cycle component $C_1$  

\begin{cla}\label{lastclaim}
    For $j\le i - 1$, we have $L_{i + 1}\ne L_{j}'$. 
\end{cla}

\begin{proof}
Suppose not. Let $i$ be the smallest integer, such that $L_{i + 1} = L_{j}'$ for some $j\le i - 1$. We claim that $y_i = x_j$ is impossible,  since at step $j$, $x_j$ has two black neighbours and neighbour $y_j$, which can not serve as $x_i$.
%because $y_j$ remains red, which would  create a non-cycle component of $H$.
Hence, $L_j'$ must have two interfaces, with $x_j$ being one of them.  By the properties of interfaces we discussed, this forces $y_i$ to be one of the blue vertices shown in the middle two graphs in Figure~\ref{interface}. 
At step $j$, $C_j'$ is the corresponding structure of $C_j$ and contains $x_j$. By Lemma~\ref{admissthird}, $y_i$ can only be the left blue vertex in the third graph in Figure~\ref{interface}; otherwise, $x_i$ and $C_j'$ would form the structure in Lemma~\ref{admissthird}, a contradiction.  However, this implies $x_i$ is a red vertex at distance 2 from $C_j'$ (which is a $T_3$),  contradicting  Lemma \ref{mainlemma}. 
\end{proof}

Since the number of path components in $H$ is finite, the above procedure must eventually terminate by Claim~\ref{lastclaim}, which contradicts $N(H) = 1$. Therefore, if $G$ is not the Petersen graph, then $N(H) = 0$. That is,  $H$ can be reorganized into a forest by a good choice of $I_1$ and $I_2$. We can then color the red vertices with two colors $2_a$ and $2_b$, such that the corresponding vertices in $H$ form two independent sets. Consequently, vertices of the same color ($2_a$ or $2_b$) maintain a pairwise distance of at least 3 in $G$. Combining with $I_1$ and $I_2$, we obtain a packing $(1,1,2,2)$-coloring of $G$.  \hfill \qed

\section{Concluding Remarks}

Gastineau and Togni~\cite{GT1} asked whether every connected subcubic graph except the Petersen graph is packing $(1,1,2,3)$-colorable. We believe it is true and therefore pose the question as a conjecture.

\begin{conj}
Every connected subcubic graph except the Petersen graph is packing $(1,1,2,3)$-colorable.    
\end{conj}

	\subsection*{Acknowledgements}
		This work was supported by the National Key Research and Development Program of China (2023YFA1010203), the National Natural Science Foundation of China (12471336, 12501473, 12401455, 12571369), and the Innovation Program for Quantum Science and Technology (2021ZD0302902). The research of X. Liu was supported by the National Natural Science Foundation of China under grant No.~12401466.

\end{document}